\documentclass{amsart}

\usepackage{amsmath, amsthm, amsfonts, amssymb, amscd}
\usepackage{enumerate}
\usepackage{framed} 
\usepackage{xcolor}
\usepackage{graphicx}
\usepackage{psfrag}
\usepackage{mathtools}
\usepackage{hyperref}
\usepackage{marginnote}
\usepackage{ulem}
\hypersetup{
    colorlinks, 
    citecolor=blue, 
    filecolor=black, 
    linkcolor=black, 
    urlcolor=black
}

\theoremstyle{definition}
\newtheorem{thm}{Theorem}[section]
\newtheorem*{thm*}{Theorem}

\newtheorem{defn}[thm]{Definition}
\newtheorem{lem}[thm]{Lemma}
\newtheorem{cor}[thm]{Corollary}

\newtheorem{prop}[thm]{Proposition}
\newtheorem{rem}[thm]{Remark}

\newcommand{\D}{\mathbb{D}}
\newcommand{\R}{\mathbb{R}}

\newcommand{\sgn}{\text{sgn}}
\numberwithin{equation}{section}

\colorlet{shadecolor}{blue!5}

\newcounter{claim}

\begin{document}

\title[Action and Winding Asymptotics. ]{Asymptotic Action and Asymptotic Winding Number 
for Area-preserving Diffeomorphisms of the disk.}


\author{David Bechara Senior}
\address{}
\curraddr{}
\email{david.becharasenior@ruhr-uni-bochum.de}
\thanks{
}

\thanks{}
\subjclass[2010]{Dynamical Systems, Symplectic Geometry }

\date{\today}

\dedicatory{}

\begin{abstract}
Given a compactly supported area-preserving diffeomorphism of the disk, we prove an integral formula relating the asymptotic action to the asymptotic winding number. As a corollary, we obtain a new proof of Fathi's integral formula for the Calabi homomorphism on the disk.
\end{abstract}

\maketitle


\section{Introduction.}

Let $\phi: \D \rightarrow \D$ be a  diffeomorphism of the disk $\D\subset \R^2$ compactly supported on the interior and preserving the standard area form $\omega_0 = dx\wedge dy$. If $\lambda$ is any primitive of $\omega_0$, we can define the \textit{action} $a_{\phi,\lambda}(z)$ of a point $z\in \D$ with respect to $\lambda$ as the value at $z$ of the unique primitive of the exact form $\phi^* \lambda - \lambda$ that vanishes near the boundary of $\D$. Equivalently, if one sees $\phi$ as the time-one map of the isotopy $\phi^t$ that is obtained by integrating the vector field coming from a compactly supported time-dependent Hamiltonian $(z,t) \mapsto H_t(z)$, the action has the expression
\[
a_{\phi,\lambda}(z) = \int_{\{t\mapsto \phi^t(z)\}} \lambda + \int_0^1 H_t(\phi^t(z))\, dt.
\]
See Section \ref{AWCs} below for precise definitions and sign conventions. In general, the value $a_{\phi,\lambda}(z)$ depends on the choice of the primitive $\lambda$, but it is independent of $\lambda$ at fixed points of $\phi$. The integral of $a_{\phi,\lambda}$ over the disk, which we denote by
\[
\mathcal{C}(\phi) := \int_\D a_{\lambda,\phi} \,\omega_0,
\]
 is also independent of the choice of $\lambda$. The real valued function $\phi\mapsto \mathcal{C}(\phi)$ is a homomorphism on $\text{Diff}_c(\D, \omega_0)$, the group of  area-preserving diffeomorphisms of $\D$ compactly supported on the interior and is named \textit{Calabi homomorphism} after \cite{Cal}.
 
 In this note, we define the \textit{asymptotic action} of $\phi$ as the limit
 \[
 a_{\phi}^{\infty}(z) := \lim_{n\rightarrow \infty} \frac{a_{\phi^n,\lambda}(z)}{n}.
 \]
Indeed, a simple application of Birkhoff's ergodic theorem implies that the above limit exists for almost every $z\in \D$ and defines a $\phi$-invariant integrable function $a_{\phi}^{\infty}$, whose integral over $\D$ coincides with the integral of $a_{\phi,\lambda}$. In particular, if $z$ is a $k$-periodic point of $\phi$, then the above limit exists and coincides with the average action of the orbit of $z$:
\[
a_{\phi}^{\infty}(z) = \frac{1}{k} \sum_{j=0}^{k-1} a_{\phi,\lambda}(\phi^j(z)).
\]
As the notation suggests, the asymptotic action $a_{\phi}^{\infty}$ is independent of the primitive $\lambda$ and hence can be expected to capture dynamical properties of $\phi$.

If $x$ and $y$ are distinct points of $\D$, the \textit{winding number} $W_{\phi}(x,y)$ is defined as the winding number of the curve
\[
[0,1] \mapsto \mathbb{S}^1, \qquad t\mapsto \frac{\phi^t(y)-\phi^t(x)}{\|\phi^t(y)-\phi^t(x)\|} = e^{i\theta(t)},
\]
i.e. 
\[
W_{\phi}(x,y):= \frac{\theta(1)-\theta(0)}{2\pi}.
\]
In the definition, we have used a Hamiltonian isotopy $\phi^t$ joining the identity to $\phi$, but the fact that the space $\text{Diff}_c(\D, \omega_0)$ is contractible implies that the value of $W_{\phi}(x,y)$ does not depend on the choice of the isotopy. The \textit{asymptotic winding number} of the pair $(x,y)$ can now be defined as the ergodic limit
\[
W_{\phi}^{\infty}(x,y) := \lim_{n\rightarrow \infty} \frac{W_{\phi^n}(x,y)}{n},
\]
which again exists for almost every pair $(x,y)$ in $\D\times \D$ and defines an integrable function $W_\phi^{\infty}$ on $\D\times \D$ that is invariant under the action of $\phi\times \phi$ and whose integral over $\D\times \D$ coincides with the integral of $W_{\phi}$.  Moreover, if $x$ is a $k$-periodic point of $\phi$ then the limit defining $W_{\phi}^{\infty}(x,y)$ exists for almost every $y\in \D$.

The main result of this note is the following formula relating the asymptotic action and the asymptotic winding number.

\begin{thm}
\label{main-thm}
If $\phi: \D \rightarrow \D$ is a compactly supported diffeomorphism of the disk $\D$ preserving the standard area form $\omega_0=dx\wedge dy$, then the identity
\[
a_{\phi}^{\infty}(x) = \int_{\D} W_\phi^{\infty} (x,y) \,\omega_0(y)
\]
holds for every $x$ in a subset $\Omega$ of $\D$ that has full measure and contains all periodic points of $\phi$.
\end{thm}

By integrating this identity over $\Omega$, and by using the fact that the integrals of the ergodic limits $a_{\phi}^{\infty}$ and $W_{\phi}^{\infty}$ agree with those of the functions $a_{\phi,\lambda}$ and $W_{\phi}$, we obtain as an immediate corollary the following result, which is originally due to Fathi \cite{FATHI}.

\begin{cor}
\label{cor:fathi}
The Calabi homomorphism on $\text{Diff}_c(\D, \omega_0)$ can be expressed by the double integral 
\[
\mathcal{C}(\phi)=\int_{\D\times\D\setminus \Delta}W_\phi(x, y) \, \omega_0(x) \wedge \omega_0(y).
\]
\end{cor}

Other proofs of this corollary are known. In \cite{GamGhys}, Gambaudo and Ghys give two algebraic proofs of it, one of which uses the fact that every homomorphism from $\text{Diff}_c(\D, \omega_0)$ into $\R$ is a multiple of the Calabi homomorphism. The latter fact is a deep result of Banyaga, see \cite{Banyaga76}. An elementary proof of Corollary \ref{cor:fathi} using complex analysis is presented by Shelukhin in \cite{Shel}.
Our  proof of Theorem \ref{main-thm} uses elementary results about intersection numbers between curves and surfaces in dimension three.

\medskip

\paragraph{\bf Aknowledgements:} I thank Jordan Payette for thorough comments on the first version of this paper. The author is supported by the SFB/TRR 191 `Symplectic Structures in Geometry,  Algebra and Dynamics',  funded by the DFG.

\section{Action,  Winding and the Calabi homomorphism}
\label{AWCs}

Denote by $\omega_0:= dx\wedge dy$ the standard area form on $\R^2$ and by $\D\subset \R^2$ the unit disk. The symbol $ \text{Diff}_c(\D, \omega_0)$ denotes  the group of 
smooth diffeomorphisms of $\D$ compactly supported on the interior and  that preserve $\omega_0$. The group $ \text{Diff}_c(\D, \omega_0)$ is contractible, see e.g. \cite{Smale59}, \cite{Banyaga}.

\subsection{Action of disk diffeomorphisms.}
Take $\phi \in \text{Diff}_c(\D, \omega_0)$ and let $\lambda$ be a smooth primitive of $\omega_0$ on $\D$. Since $\phi$ preserves $\omega_0$,  the 1-form $\phi^* \lambda - \lambda$ is closed and hence exact on $\D$.
\begin{defn}

The {\em action} of $\phi$ with respect to $\lambda$ is the unique smooth function  
\[
a_{\phi, \lambda} : \D \rightarrow \R
\]
such that
\begin{equation}
\label{uno}
da_{\phi,\lambda}  = \phi^* \lambda - \lambda,  
\end{equation}
and 
\begin{equation}
\label{due}
a_{\phi,\lambda} (z) = 0, 
\end{equation}
for every $z\in \partial \D$. 
\end{defn}

The existence of the function $a_{\phi,\lambda}$ satisfying (\ref{uno}) and (\ref{due}) follows from the fact that any primitive of $\phi^*\lambda-\lambda$ is constant on $\partial \D $ because $\phi_{|\partial \D}= \mathrm{id}.$ 

\begin{lem}
\label{formule}
Let $ { \phi}$ and $ {\psi}$ be elements of $\mathrm{Diff}_c(\D, \omega_0)$. Let $\lambda$ be a smooth primitive of $\omega_0$ and let $u$ be a smooth real function on $\D$. Then:
\begin{enumerate}[(i)]
\item $a_{ { \phi}, \lambda+du} = a_{ { \phi}, \lambda} + u\circ  \phi - u$.
\item $a_{ {\psi}\circ  { \phi}, \lambda} = a_{ {\psi}, \lambda} \circ  \phi + a_{ { \phi},  \lambda} = a_{ {\psi}, \lambda} + a_{ { \phi}, \psi^* \lambda}$.
\item $a_{ { \phi}^{-1}, \lambda} = - a_{ { \phi}, \lambda} \circ  \phi^{-1} = - a_{ { \phi}, ( \phi^{-1})^*\lambda}$.
\end{enumerate}
\end{lem}

\begin{proof}
The first claim follow from the identities
\[
 \phi^* (\lambda + du) - (\lambda + du) = d a_{ { \phi}, \lambda} +  \phi^* (du) - du = d(a_{ { \phi}, \lambda}+u\circ  \phi - u), 
\]
and
\[
 a_{ { \phi}, \lambda}(z) + u( \phi(z)) - u(z)=u( \phi(z)) - u(z)=0, \qquad \forall z\in\partial\D.
\]
The first identity in (ii) follows from
\[
(\psi\circ  \phi)^* \lambda - \lambda =  \phi^*(\psi^* \lambda - \lambda) +  \phi^* \lambda - \lambda =  \phi^*(da_{ {\psi}, \lambda}) + d a_{ { \phi}, \lambda} = d ( a_{ {\psi}, \lambda} \circ  \phi + a_{ { \phi}, \lambda} ), 
\]
and
\[
a_{ { \phi}, \lambda} (z) + a_{ {\psi}, \lambda} ( \phi(z))=0, \qquad  \forall z\in\partial\D.
\]
 The second identity in (ii) follows from
\[
(\psi\circ  \phi)^* \lambda - \lambda =  \phi^*(\psi^* \lambda) - \psi^* \lambda + \psi^* \lambda - \lambda = d( a_{ { \phi}, \psi^* \lambda} + a_{ {\psi}, \lambda}), 
\]
and
\[
\begin{split}
 a_{ {\psi}, \lambda}(z) + a_{ { \phi}, \psi^* \lambda}(z)=0,  \qquad \forall z\in\partial\D.
\end{split}
\]
The two formulas in (iii) follow from those in (ii) applied to the case $ {\psi}= { \phi}^{-1}$,  because $a_{\mathrm{id}, \lambda}=0$ for every $\lambda$.
\end{proof}

\begin{rem}
Statement (i) in the lemma above implies that the value of $a_{\phi,\lambda}(z)$ is independent of the primitive $\lambda$ if $z\in \D$ is a fixed point of $\phi$.
\end{rem}

It is convenient to read the action also in terms of  the Hamiltonian formalism. Let $H:\D\times[0,1] \rightarrow \R$,  $H_t(z)=H(z,t)$ be  a smooth time-dependent Hamiltonian that vanishes near $\partial \D \times [0,1]$. Consider  its  time-dependent Hamiltonian vector field $X_t$, which is  defined by the condition 
\begin{equation}
\label{hamilt}
 \omega_0(X_{t}, \cdot) = dH_t(\cdot).
\end{equation}
We denote by $\phi^t$ its non-autonomous flow, i.e.\ the isotopy that is defined by
\[
\frac{d}{dt}\phi^t=X_{t}\circ \phi^t, \qquad \phi^0 = \mathrm{id}.
\]
Any $\phi$ in  $\text{Diff}_c(\D, \omega_0)$ is the time-one map $\phi=\phi^1$ of an isotopy $\phi^t$ as above, which we shall call Hamiltonian isotopy.

\begin{prop}
If $\phi$, $\phi^t$ and $H$ are as above, then
\begin{equation}\label{actionH}
    a_{\phi, \lambda}(z)=\int_{ \{ t\rightarrow \phi^t(z) \} }\lambda + \int_0^1 H_t(\phi^t(z))dt. 
\end{equation}
\end{prop}

\begin{proof}
Let $\gamma:[0, 1]\rightarrow \D$ be a smooth path such that $\gamma(0)=z$ and $\gamma(1)\in \partial \D$. Define the smooth map
\[
\psi: [0, 1]^2 \rightarrow \D,  \qquad \psi(s, t) := \phi^t(\gamma(s)).
\]
We compute the integral of $\psi^*\omega_0$ on $[0, 1]^2$ in two different ways. The first computation gives us:
\[
\begin{split}
\int_{[0, 1]^2} \psi^* \omega &= \int_{[0, 1]^2} \omega_0(\phi^t(\gamma(s)))\left[ \frac{\partial}{\partial s} \phi^t(\gamma(s)),  X_{t} (\phi^t(\gamma(s))) \right] \,  ds \,  dt \\ &=-  \int_{[0, 1]^2} dH_t(\phi^t(\gamma(s)))\left[ \frac{\partial}{\partial s} \phi^t(\gamma(s)) \right]\,  ds\,  dt \\ &=-  \int_0^1 \left( \int_0^1 \frac{\partial}{\partial s} H_t(\phi^t(\gamma(s)))\,  ds \right)\,  dt \\ &= - \int_0^1 \left( H_t(\phi^t(\gamma(1))) - H_t(\phi^t(z)) \right)\,  dt = \int_0^1 H_t(\phi^t(z))\,  dt.
\end{split}
\]
The second computation uses Stokes' theorem and the fact that ${\phi^t}|_{\partial\D}= \mathrm{id}$:
\[
\begin{split}
\int_{[0, 1]^2} \psi^* \omega_0 &= \int_{[0, 1]^2} \psi^* d\lambda = \int_{[0, 1]^2} d\psi^* \lambda = \int_{\partial [0, 1]^2} \psi^* \lambda \\ &= \int_{\gamma} \lambda + \int_{\{t\mapsto \phi^t(\gamma(1))\}} \lambda - \int_{\phi\circ \gamma} \lambda - \int_{\{t\mapsto \phi^t(z)\}}\lambda  \\ &=  - \int_{\gamma} ( \phi^* \lambda - \lambda) - \int_{\{t\mapsto \phi^t(z)\}} \lambda\\ &= - \int_{\gamma} da_{\phi, \lambda} - \int_{\{t\mapsto \phi^t(z)\}} \lambda\\ &= a_{\phi, \lambda}(z) - \int_{\{t\mapsto \phi^t(z)\}}\lambda.
\end{split}
\]
The desired formula for $a_{\phi, \lambda}(z)$ follows by comparing the above two identities.

\end{proof}

\subsection{The Calabi Homomorphism.}

\begin{defn}
The Calabi homomorphism is the map
\[
    \mathcal{C}:\text{Diff}_c(\D, \omega_0) \rightarrow  \R
\]
defined by 
\[
\mathcal{C}(\phi)=\int_{\D}a_{\phi,  \lambda}\, \omega_0.
\]

\end{defn}

The first thing we remark is that this map is well defined, meaning that it does not depend on the choice of the primitive of $\omega_0$. This follows  from Lemma \ref{formule}(i) and the fact that $\phi$ is an area preserving diffeomorphism.

Next we should prove that it is in fact a homomorphism. This is the following:
\begin{prop}
If $ \phi, \psi \in \text{Diff}_c(\D, \omega_0)$ then  $\mathcal{C}(\phi \circ \psi)=\mathcal{C}(\phi)+\mathcal{C}(\psi)$.
\end{prop}

\begin{proof}

From Lemma \ref{formule}(ii) we have $$a_{\phi \circ \psi,  \lambda}=a_{\psi, \phi^*\lambda}+a_{\phi, \lambda}, $$ and  integrating on both sides we get the desired result. 

\end{proof}

\begin{rem}
We can express $\mathcal{C}(\phi)$ also in terms of a defining Hamiltonian. If $H: \D\times [0,1] \rightarrow \R$ is a time-dependent Hamiltonian vanishing near $\partial \D\times [0,1]$ such that the time-one map of the corresponding Hamiltonian isotopy is $\phi$, then $\mathcal{C}(\phi)$ is given by the formula
\[
\mathcal{C}({\phi}) = 2 \int_{\D \times [0,1]} H(z,t)\,  \omega_0 \wedge dt.
\]
The proof of this equality can found in \cite{MS}[Lemma 10.27] or with the above notation in \cite{ABHS3}[Prop 2.7] in the more general setting of Hamiltonian diffeomorphisms of $\D$ that are not necessarily compactly supported. It can be used  for instance,  to show that the Calabi homomorphism  is not continuous in the $\text{C}^0$-topology. This can be done by constructing radially defined Hamiltonians using bump functions with fixed non-zero  integral and whose time one-maps converge to the identity, as shown in \cite{GamGhys}.
\end{rem}

\subsection{Winding and Intersection numbers for disk Diffeomorphisms.} \label{wiidd}
Let $\phi\in \text{Diff}_c(\D, \omega_0)$  and let $\phi^t$ be a Hamiltonian isotopy  such that $\phi^0=\mathrm{id}$ and $\phi^1=\phi$. 

\begin{defn}
The \textit{winding number} $W_\phi(x, y)$ of a pair of distinct  points $x,y\in \D$ is the real number
\[
W_\phi(x, y):=\frac{{\theta}(1)-{\theta}(0)}{2\pi},
\]
where $\theta: [0,1] \rightarrow \R$ is a continuous function such that
\begin{equation}\label{star}
 \frac{\phi^t(y)-\phi^t(x)}{||\phi^t(y)-\phi^t(x)||}=e^{i\theta(t)}.
\end{equation}
\end{defn}

The value of $W_\phi(x, y)$ does not depend on the isotopy joining $\phi$ to the identity. This holds because $\text{Diff}_c(\D,\omega_0)$ is contractible and in particular any two paths joining $\phi$ to the identity are homotopic, which implies that the winding number is the same. The map $W_\phi$ is smooth and bounded on $(\D\times\D)\setminus \Delta$, where $\Delta$ denotes the diagonal $\{(x, x)|x\in \D \} $ in $\D\times\D$ (see \cite{GamGhys}).

We want to define another relevant quantity and for this we need to work with intersection numbers in dimension three. 

To begin with, consider an embedded compact oriented curve $\Gamma \subset \D\times [0,1]$. If $ S \subset \D\times [0,1]$ is a compact co-oriented surface, that is, the normal bundle $\mathcal{N}S$ is oriented, then an intersection point $p\in S\cap \Gamma$ is called transverse  if  $$ T_p \Gamma \oplus T_p S=T_p (\D\times [0,1]). $$
A transverse intersection point $p$ is called positive if the orientation of $T_p\Gamma$ coincides with the orientation of $\mathcal{N}_p S$ and otherwise is called negative. Assume that $\Gamma$ is everywhere transverse to $S$. Then the set $\Gamma\cap S$ is finite and we denote by $I_+(\Gamma,S)$ the number of positive intersections and $I_-(\Gamma,S)$ the number of negative intersections. The difference of these two numbers defines the intersection number of $\Gamma$ and $S$:
\[
I(\Gamma,S) := I_+(\Gamma,S) - I_-(\Gamma,S) \in \mathbb{Z}.
\]

Let  $\phi\in \text{Diff}_c(\D, \omega_0)$ and $\{\phi^t\}_{t\in [0, 1]}$ be a Hamiltonian isotopy  that connects the identity to $\phi$. Given $x\in \D$, we denote by $\Gamma_{\phi}(x)$ the embedded curve
\[
\Gamma_{\phi}(x) := \{ (\phi^t(x),t) \mid t\in [0,1] \} \subset \D \times [0,1]
\]
with its natural orientation. Fix an element  $e\in \partial\D$ and for every $x \text{ in the interior of } \D$  consider the compact surface
\[
S^e_\phi(x)\subset \D\times [0,1]
\]
that is obtained by connecting each point $(\phi^t(x),t)$ of $\Gamma_{\phi}(x)$ to the boundary point $(e,t)$ by a line segment. In other words, $S^e_\phi(x)$ is parametrized by the smooth embedding
\begin{align}{\label{reparS}}
    \Psi :[0, 1] \times  [0, 1] &\rightarrow \D  \times [0, 1],
    &(t,r) \mapsto (\phi^t(x) + r(e-\phi^t(x)), t)
\end{align}
{
which induces an orientation on $S$.   We consider the disk and real line as having their canonical orientations and the ambient space $\mathbb{D}\times[0,1]$ the induced product orientation.  
}
The vector 
\begin{equation}
\label{co-orientation}
(i(e-\phi^t(x)),0 )\in T_{\Psi(t,r)}(\D\times[0,1]),
\end{equation}
where $i$ denotes the counterclockwise rotation by $\pi/2$ on $\R^2$, is everywhere transverse to $S_\phi^e(x)$ and hence defines a co-orientation of this surface.   
{We note that this co-orientation coincides with that induced by the above mentioned orientations on $S_\phi^e(x)$ and 
$\mathbb{D}\times[0,1]$:  Indeed, it suffices to check this at a single point, and we observe that at the point $(e,0)\in S_\phi^e(x)\subset\mathbb{D}\times[0,1]$, 
the basis $\{\partial_t\Psi, \partial_r\Psi, \big(i(e-\phi^t(x)),0\big)\}$ equals $\{(0,1),\big(e-\phi^t(x),0\big),\big(i(e-\phi^t(x)),0\big) \}$ which has the same orientation as $\{\big(e-\phi^t(x),0\big),\big(i(e-\phi^t(x)),0\big),(0,1)\}$ which is clearly positive.   
} 

If $x$ is an interior point of $\D$ and $y\neq x$ is another point in $\D$, then we define the intersection number of $x \text{ and } y$ as
\[
I^e_\phi(x, y):=  I(\Gamma_{\phi}(y), S^e_\phi(x)),
\]
whenever $\Gamma_{\phi}(y)$ meets $S^e_\phi(x)$ transversely.   
\begin{figure}[ht]
    \centering
   \scalebox{.8}{\includegraphics{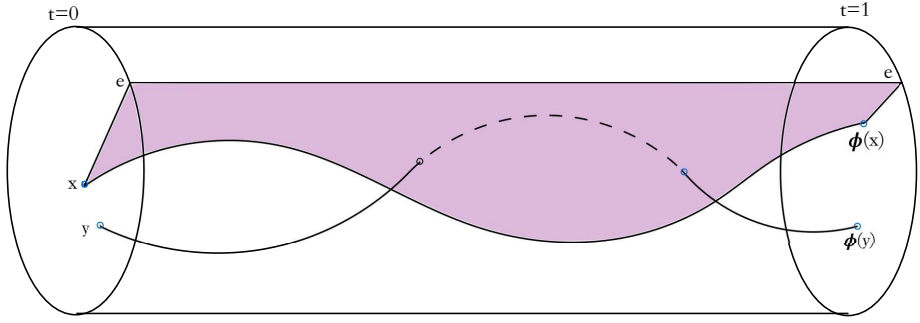}}
    \caption{$I(\Gamma_{\phi}(y), S^e_\phi(x))=0$.}
    \label{fig:my_label}
\end{figure}\\

\begin{rem} Given $x\neq y$ as above, it follows from Sard's theorem  that the set of points $e\in \partial\D$ for which $\Gamma_{\phi}(y)$ is transverse to $S^e_{\phi}(x)$ has full Lebesgue measure. Furthermore, even though the definition above depends on the choice of the element $e\in \partial\D$, a different choice of $e$ changes the value of $I_{\phi}^e(x,y)$ by at most one. This follows from Proposition \ref{Bound1} below.
\end{rem}

Next we relate the winding number and the intersection number. 

\begin{prop}\label{Bound1}
Let $x\in \mathrm{int}(\D)$, $y\in \D\setminus \{x\}$ and $e\in \partial \D$ be such that $\Gamma_{\phi}(y)$ meets $S^e_{\phi}(x)$ transversally. Then
\begin{equation}\label{ineq1}
    \Big|W_\phi(x, y)-I^e_\phi(x, y)\Big|\leq \frac{3}{2}
\end{equation}
\end{prop}

{
\begin{rem} 
Although we will not use this fact, the constant $\frac{3}{2}$ in (\ref{ineq1}) is optimal.  
\end{rem}
}

\begin{proof} 
For every $t\in [0,1]$, we can push out the trajectory of $y$ in $\D\times [0,1]$ to the boundary of the cylinder in the following way: consider the curve
$\beta:[0,1]\rightarrow \partial\D \times[0,1]$ defined by  $\beta(t)=(h_t(\phi^t(y)),t) $ where $h_t:\D\to \partial\D$ maps  each ray from $\phi^t(x)$ to its intersection with $\partial\D$.  More precisely, every $z\neq \phi^t(x)$ in $\D\setminus\{\phi^t(x)\}$ can be written as $z=\phi^t(x)+r e^{i\alpha}$ where 
\[
r=|z-\phi^t(x)| \text{ and } e^{i\alpha}:=\frac{z-\phi^t(x)}{|z-\phi^t(x)|},
\]
and we set
\[
h_t(z):=\phi^t(x)+R e^{i\alpha},  
\]
where 
\[
R =\max\{s>0|\, ||\phi^t(x)+se^{i\alpha(z) }|| \leq 1 \}.
\]
The winding of this new trajectory on $\partial\D$  is given by $$w_\phi(x,y)=\frac{\sigma(1)-\sigma(0)}{2\pi}$$ 
where $$e^{i \sigma(t)}=h_t(\phi^t(y)).$$ 
Since the absolute value of the angle between the vectors $\phi^t(y)-\phi^t(x)$ and $h_t(\phi^t(y))$ is not larger than $\frac{\pi}{2}$, the lifts $\sigma$ as above and $\theta$ as in (\ref{star}) can be chosen so that $|\theta(t)-\sigma(t)|\leq\frac{\pi}{2} $. It follows that 
\begin{equation}\label{twowindings}
    |w_\phi(x,y)-W_\phi(x,y)|\leq \frac{1}{2}.
\end{equation}

\medskip
Next, we consider  the curve $
    \eta:[0, 1]  \rightarrow \partial\mathbb{D}$, 
\[ 
t \rightarrow \eta(t):=h_t(\phi^t(y)),
\]
and extend it to a closed curve $\hat{\eta}: [0,2] \rightarrow \partial\mathbb{D}$ in such a way that $\hat{\eta}((1,2))$ does not contain $e$. The absolute value of the winding number of the curve $\hat{\eta}|_{[1,2]}$ does not exceed one.
We now lift $\hat{\eta}$ to a self-map of $\partial\mathbb{D}$:
\[
    f:\partial\mathbb{D} \rightarrow \partial\mathbb{D}, \qquad e^{ 2 \pi i t} \mapsto \hat{\eta}(2t).
\]

The oriented degree of this map coincides with its winding number, which differs from the winding number $w_\phi(x,y)$ exactly by the winding number of $\hat{\eta}|_{[1,2]}$, and hence
\[
|w_\phi(x,y)-\deg(f)|\leq 1.
\]
On the other hand, our transversality assumption implies that $e$ is a regular value for the map $f$, so the oriented degree of this map can be computed as
\[
\deg(f)=\sum_{q\in f^{-1}(e)}\sgn(f'(q)),
\]
but the latter sum gives us precisely $I_{\phi}^e(x,y)$. This fact, together with (\ref{twowindings}), yields (\ref{ineq1}).
\end{proof}

In order to state the next result, we need to fix some notation and recall some basic facts about the oriented degree for smooth maps between surfaces. Let  $M$ and $N$ be two oriented surfaces (possibly with boundary), with $M$ compact, and let $f:M\to N$ be a smooth map. We set
\[
N_{\text{reg}}(f) :=\{y\in N |\, y \text{ is a regular value of } f\}.
\]
Then $N\setminus{N_{\text{reg}}(f)}$ has measure zero. For every $y\in N_{\text{reg}}(f)\setminus f(\partial M)$ the set $f^{-1}(y)$ is finite and the degree of $f$ relative to $y$ is the integer
\begin{equation}\label{Def:degree}
\deg(f,M,y)=\sum_{x\in f^{-1}(y)}\sgn \det df(x) \in \mathbb{Z}.
\end{equation}
Furthermore, the function $y\mapsto\deg(f,M,y)$ extends to a locally constant function on $N\setminus{f(\partial M)}$, and for $\eta$ any $2$-form on $N$ we have the identity
\begin{equation}\label{degreeProperty}
    \int_M f^*\eta=\int_N \deg(f,M,y)\, \eta(y).
\end{equation}

We shall now show that $I_\phi^e(x,y)$ coincides with the degree of a  suitable map between two oriented surfaces. Let $x\in \D$ and $e\in\partial\D$ be as above and  consider  the smooth map
\[
    \rho:S_\phi^e(x) \to        \D, \qquad \rho(z,t):=(\phi^t)^{-1}(z).
\]
Note that the boundary of $S_\phi^e(x)$ consists of the two curves $\Gamma_{\phi}(x)$ and $\Gamma_{\phi}(e)=\{e\}\times [0,1]$ and of the line segments $[x,e]\times \{0\}$ and $[\phi(x),e]\times \{1\}$. The curves $\Gamma_{\phi}(x)$ and $\Gamma_{\phi}(e)$ are mapped into $x$ and $e$ by $\rho$, whereas $[x,e]\times \{0\}$ and $[\phi(x),e]\times \{1\}$ are mapped into $[x,e]$ and $\phi^{-1}([\phi(x),e])$, respectively. We conclude that
\[
\rho(\partial S_{\phi}^e(x)) = [x,e] \cup \phi^{-1}([\phi(x),e]).
\]
In the following result, the disk $\D$ is given its standard orientation and the surface $S_\phi^e(x)$ the orientation that is induced by the parametrization (\ref{reparS}), meaning that
\begin{equation}
\label{oriented-basis}
\partial_t \Psi(t,r) = \bigl( (1-r) X_t(\phi^t(x)),1\bigr), \qquad \partial_r \Psi(t,r) = \bigl( e-\phi^t(x),0 \bigr)
\end{equation}
is a {positively} oriented basis of $T_{\Psi(t,r)} S_\phi^e(x)$, for every $(t,r)\in [0,1]\times [0,1]$.   {With respect to these orientations the degree of the map $\rho$ is well 
defined and at regular values is given by the formula (\ref{Def:degree}).    Recall that the intersection number $I_\phi^e(x,y)$ of the curve $\Gamma_{\phi}(y)$ with the surface $S_\phi^e(x)$ 
is defined with respect to the co-orientation in (\ref{co-orientation}).}

\begin{lem}\label{transverseregular}
The point $y\in \D$ is a regular value for $\rho$ if and only if the curve $\Gamma_\phi(y)$ meets $S_\phi^e(x)$ transversally. 
{Moreover, whenever $y$ is a regular value for $\rho$,} we have the identity 
\begin{equation}
    \deg(\rho,S_\phi^e(x),y)=I_\phi^e(x,y).  
\end{equation}
\end{lem}
\begin{proof}
Set for simplicity $S:= S_\phi^e(x)$. The map $\rho$ is the restriction to the surface $S$ of the map
\[
\tilde{\rho} : \D \times [0,1] \rightarrow \D, \qquad (z,t) \mapsto (\phi^t)^{-1}(z).
\]
The restriction of $\tilde{\rho}$ to $\D \times \{t\}$ is a diffeomorphism for every $t\in [0,1]$ and hence $\tilde{\rho}$ is a submersion. The inverse image of each $y\in \D$ by $\tilde{\rho}$ is the curve
\[
\tilde{\rho}^{-1}(y) = \Gamma_{\phi}(y),
\] 
and hence if $\tilde{\rho}(z,t)=y$ we have
\begin{equation}
\label{line}
\ker d\tilde{\rho}(z,t) = T_{(z,t)} \Gamma_\phi(y) = \R \tilde{X}(z,t),
\end{equation}
where 
\[
\tilde{X}(z,t) := \bigl( X_t(z),1 \bigr).
\]
These considerations imply that $(z,t)\in S$ belongs to $\rho^{-1}(y)$ if and only if $\Gamma_\phi(y)$ and $S$ meet at $(z,t)$.    Fix such a point $(z,t)\in \Gamma_\phi(y)\cap S$.    
Then the range of $d\rho(z,t)$ equals the range of $d\tilde{\rho}(z,t)|_{T_{(z,t)}S}$ which is surjective if and only if it coincides with the range of $d\tilde{\rho}(z,t)$, 
which, since the latter is a submersion, is the case if and only if the kernel of $d\tilde{\rho}(z,t)$ is transverse to $T_{(z,t)}S$.   From (\ref{line}), this shows that $y\in \D$ is a regular value for 
$\rho$ if and only if the curve $\Gamma_\phi(y)$ meets $S$ transversally.   

To prove the second part of the Lemma, we fix a regular value $y\in \D$ for $\rho$ and we fix a point
\[
(z,t)\in \rho^{-1}(y) = \Gamma_\phi(y) \cap S.
\]
We wish to prove that this intersection point is positive if and only if $d\rho(z,t)$ is orientation preserving.  Once this is proven, the last claim of the Lemma follows from the formula 
(\ref{Def:degree}) for the degree with respect to a regular value.   

Recall $\Psi$ the parametrization of $S$ given by (\ref{reparS}).   This gives us the basis 
\[
					(v,1):=\partial_t\Psi(t,r)=\Big((1-r)X_t(\phi^t(x)),1\Big)\qquad (u,0):=\partial_r\Psi(t,r)=\Big(e-\phi^t(x),0\Big)
\]
for the tangent space $T_{\Psi(t,r)}S$, while  $T_{\Psi(t,r)}\Gamma_\phi(y)$ is generated by the non-zero vector
\[
(w,1):=\Big(X\big(\phi^t(x) + r(e-\phi^t(x))\big),1\Big).
\]

The co-orientation of $S$ at $(z,t)$ is given as in (\ref{co-orientation}) by the vector $(iu,0)$.   The intersection at $\Psi(t,r)$ is by definition positive, precisely 
when the two co-orientations $(iu,0)$ and $(w,1)$ of $T_{\Psi(t,r)}S$ agree.   In other words, precisely when 
the two bases $\big((v,1),(u,0),(w,1)\big)$ and $\big((v,1),(u,0),(iu,0)\big)$ determine the same orientation for $ T_{(z,t)}{(\D\times[0,1])}$.  
{By the discussion following (\ref{co-orientation}) the latter is the standard orientation on $\D\times[0,1]$.}   
Hence it is enough to prove that $\big((v,1),(u,0),(w,1)\big)$ is a positive basis if and only if $d\rho(z,t)$ is orientation preserving with respect to the basis 
$\big((v,1),(u,0)\big)$ for $T_{\Psi(t,r)}S$.    The basis $\bigl((v,1),(u,0),(w,1)\bigr)$ is positive if and only if the following determinants are positive
\[
 \det\left(  \begin{array}{cc} v  & 1  \\ u  & 0 \\ w & 1 \end{array} \right) =  \det \left(  \begin{array}{cc} v-w  & 0 \\ u & 0 \\ w & 1 \end{array} \right) =  \det \left(  \begin{array}{cc} v-w   \\ u  \end{array} \right).
\]
Since $\phi$ is isotopic to the identity, so $d(\phi^t)^{-1}(z,t)$ must be symplectic and therefore
\[
\det \left(  \begin{array}{cc} v-w   \\ u  \end{array} \right) = \det \left(  \begin{array}{cc}  d(\phi^t)^{-1}(z,t)[v-w]   \\  d(\phi^t)^{-1}(z,t)[u]  \end{array} \right)= \det d\hat{\rho}(t,r)
\]
where $\det d\hat{\rho}(t,r)$ is taken with respect to the standard basis of $\R^2$ and the basis $\partial_t,\partial_r$ of $T_{(t,r)}[0,1]^2$.
The last equality above follows from the computation of the tangent map $d\hat{\rho}(t,r):T_{(t,r)}[0,1]^2 \to T_y \D$ as 
\begin{align*}
    d\hat{\rho}(t,r)[\partial_{r}]&=d(\phi^t)^{-1}\big(\phi^t(x) + r(e-\phi^t(x))\big)[e-\phi^{t}(x)]\\
    &=d(\phi^t)^{-1}(z,t)[u],\\
    d\hat{\rho}(t,r)[\partial_t]&=-X\left((\phi^t)^{-1}\big(\phi^t(x) + r(e-\phi^t(x))\big)\right)\\
    &\hspace{50pt}+(1-r)d(\phi^t)^{-1}\big(\phi^t(x) + r(e-\phi^t(x))\big)[X(\phi^{t}(x))]\\
                &=d(\phi^t)^{-1}\big(\phi^t(x) + r(e-\phi^t(x))\big)[-X(\phi^t(x) + r(e-\phi^t(x))\\
                &\hspace{183pt} +(1-r)X(\phi^{t}(x))]\\
                &= d(\phi^t)^{-1}(z,t)[v-w].
\end{align*}
Since $d\hat{\rho}(t,r)$ is orientation preserving if and only if $d\rho(z,t)$ also is, we conclude $\tilde{X}(z,t)$  positively intersects $S$ at $(z,t)$ if and only if  $d\rho(z,t)$ is orientation preserving.
\end{proof}
\begin{rem} 
\label{ae}
The above lemma and Sard's theorem imply that $\Gamma_\phi(y)$ meets $S_\phi^e(x)$ transversally for almost every $y\in \D$. Together with the properties of the oriented degree mentioned above, this lemma  implies also that the function $I_\phi^e(x,\cdot)$ extends to a locally constant function on $\D \setminus ([x,e] \cup \phi^{-1}([\phi(x),e]))$. Of course, this extension could also be obtained by defining the intersection number of an oriented curve and a co-oriented surface without assuming transversality, but just the condition that there are no intersections on either of the two boundaries, by the usual perturbation argument. Actually, we shall not need this extension of the function $y\mapsto I_{\phi}^e(x,y)$: It will be enough to know that it is well-defined for almost every $y\in \D$ and that the bound of Proposition \ref{Bound1} and the identity of Lemma \ref{transverseregular} hold.
\end{rem} 

\section{Asymptotic Action \& Asymptotic winding number.}
In this section we define the asymptotic versions of the action and winding number introduced in the previous section.These definitions build on Birkhoff's ergodic theorem, which we shall use in the following form.

\begin{thm*}[Birkhoff's Ergodic Theorem]
If $\varphi$ is an endomorphism of a finite measure space $(\Omega, \mathcal{A}, \mu)$ and if $f\in 
L^1(\Omega,\mathcal{A},\mu)$, then the averages 
\[
A_n f=\frac{1}{n}\sum_{i=0}^{n-1}f\circ \varphi^i
\]
converge $\mu-$a.e. and in $L^1(\Omega,\mathcal{A},\mu)$ to some $\varphi$-invariant function $\Bar{f}$. Furthermore for each $\varphi$-invariant $A\in \mathcal{A}$ 
\begin{equation}
    \int_{A}\Bar{f}d\mu=\int_A fd\mu.
\end{equation}
\end{thm*}
\subsection{Asymptotic Action}
In our setup we will work with the space $(\D, \mathcal{B}, \mu)$ where $\mathcal{B}$ is the Borel $\sigma$-algebra on $\D$ and $\mu$ is the Lebesgue measure.
\begin{defn}
Let $\phi \in \text{Diff}_c(\D, \omega_0)$ and let $\lambda$ be a smooth primitive of $\omega_0$ on $\D$. With $a_{\phi, \lambda}$ as presented before  we define the {\em asymptotic action} of $\phi$ with respect to $\lambda$ as the limit  
\[
a^\infty_{\phi, \lambda}(z)=\lim_{n\rightarrow\infty}\frac{a_{\phi^n, \lambda}(z)}{n}.
\]
By Lemma \ref{formule}(ii) we can write 
\begin{equation}
\label{AdditiveAction}
\frac{a_{\phi^n, \lambda}(z)}{n}=\frac{1}{n}\sum_{j=0}^{n-1}a_{\phi, \lambda}(\phi^j(z))
\end{equation}
and then taking $n\to \infty$ and applying Birkhoff's ergodic theorem  we can ensure that the limit $a^\infty_{\phi, \lambda}(z)$ exists for almost every $z\in \D$. Furthermore, Birkhoff's theorem guarantees that 
\begin{equation}\label{BirkAction}
    \int_{\D}a^\infty_{\phi, \lambda}(z)\, \omega_0(z)=\int_\D a_{\phi, \lambda}(z)\, \omega_0(z).
\end{equation}
\end{defn}
The first thing to remark is that the asymptotic action does not depend on the primitive $\lambda$ of $\omega_0$.
\begin{prop}
The asymptotic action $a^\infty_{\phi, \lambda}$ is independent of the primitive $\lambda$ of $\omega_0$.
\end{prop}
\begin{proof}
Let $\lambda+du$ be another primitive of $\omega$. Then by Lemma \ref{formule}(i) the action with respect to this primitive is
$$a_{\phi, \lambda+du}=a_{\phi, \lambda}+u\circ \phi-u.$$
We can now check the Birkhoff sum up to a finite order
$$\sum_{j=0}^{n-1} a_{\phi,  \lambda+du}(\phi^j(z))=\sum_{j=0}^{n-1} a_{\phi,  \lambda}(\phi^j(z)+\sum_{j=0}^{n-1} (u\circ \phi -u)(\phi^j(z)).$$
Note that the last term on the right hand side is a telescopic sum of the form $$\sum_{j=0}^{n-1} \bigl( u(\phi^{j+1}(z))-u(\phi^{j}(z)) \bigr) =u(\phi^{n}(z))-u(z), $$
and this is uniformly bounded for all $n$. Therefore when taking the Birkhoff average,  this last term goes to zero i.e. 
$$\lim_{n\rightarrow \infty}\frac{1}{n}\sum_{j=0}^{n-1} a_{\phi,  \lambda}(\phi^j(z))=\lim_{n\rightarrow \infty}\frac{1}{n}\sum_{j=0}^{n-1} a_{\phi,  \lambda+du}(\phi^j(z)).$$
This shows that $a_{\phi, \lambda}^\infty(z)$ does not depend on the primitive $\lambda$. More precisely, the existence and value of the limit of $a_{\phi^n,\lambda}(z)/n$ do not depend on the primitive $\lambda$.
\end{proof}

This lets us write from now on the asymptotic action as $a_\phi^\infty$. It readily follows from (\ref{AdditiveAction}) that if $z$ is a $k$-periodic point of the map $\phi$, then the asymptotic action of $\phi$ coincides with the average action of the orbit of $z$:
\begin{equation}
\label{actionPeriodic}
    a_{\phi}^{\infty}(z) = \frac{1}{k} \sum_{j=0}^{k-1} a_{\phi,\lambda}(\phi^j(z)).
\end{equation}
The average action of periodic points is the object of study of \cite{Hut}, in which Hutchings proves the existence of periodic points of $\phi$ whose action satisfies suitably bounds in terms of $\mathcal{C}(\phi)$.

\subsection{Asymptotic  Winding number.}\label{awn}
In the same way as with the action we now consider the asymptotic winding number.
\begin{defn}
The {\em asymptotic winding number} of $\phi\in \text{Diff}_c(\D, \omega_0)$ is the limit 
\[
W^\infty_\phi(x, y)=\lim_{n\rightarrow\infty} \frac{W_{\phi^n}(x, y)}{n}.
\]

\end{defn}
To guarantee its existence we  can also  use Birkhoff's ergodic theorem. This time we take $$(Q:=\D\times \D ,\, \mathcal{B}(\D\times\D),\, \mu_Q :=\mu\oplus \mu)$$ with the action of $\phi_Q:=(\phi, \phi) $. For $x\neq y$ in $\D$ the winding number satisfies 
\[
W_{\phi^n}(x, y)=\sum_{j=0}^{n-1}W_\phi(\phi_Q^j(x, y)), 
\]
and hence the time average 
\begin{align*}
   W_\phi^\infty(x, y)=&\lim_{n\rightarrow \infty}\frac{W_{\phi^n}(x, y)}{n}\\=&\lim_{n\rightarrow \infty}\frac{1}{n}\sum_{j=0}^{n-1}W_\phi(\phi_Q^j(x, y))
\end{align*}
exists for $\mu_Q$-a.e $(x, y)\in \D\times\D $.  Furthermore  Birkhoff's theorem ensures that the time and space averages coincide i.e.
\begin{equation}
\label{intW}
    \int_{\D\times\D}W^\infty_\phi(x, y)\, \omega_0(x) \wedge \omega_0(y)=\int_{\D\times\D}W_\phi(x, y) \, \omega_0(x) \wedge \omega_0(y).
\end{equation}
In particular $W_\phi^\infty$ is an integrable function and hence by the theorem of Fubini-Tonelli we obtain that the integral
\begin{equation}\label{windingIntegral}
    \int_{\D}W_\phi^\infty(x, y) \, \omega_0(y)
\end{equation}{}
is well-defined for almost every $x\in \D$ and defines an integrable function of $x$.

\begin{rem}
\label{periodic}
If $x$ is a $k$-periodic point of $\phi$, then the sequence of functions
\[
y \mapsto \frac{W_{\phi}^n(x,y)}{n}
\]
converges to $W_{\phi}^{\infty}(x,y)$ for almost every $y\in \D$ and in $L^1(\D)$. Indeed, the identity
\[
W_{\phi^{kh}}(x,y) = \sum_{j=0}^{h-1} W_{\phi^k}(\phi^{kj}(x),\phi^{kj}(y)) = \sum_{j=1}^{h-1} W_{\phi^k}(x,\phi^{kj}(y))
\]
shows that this function of $y$ is the Birkhoff sum of the function $y\mapsto W_{\phi^k}(x,y)$ with respect to the map $\phi^k$. By Birkhoff's ergodic theorem, there is an integrable function $w:\D \rightarrow \R$ such that
\[
\frac{W_{\phi^{kh}}(x,y)}{h} \rightarrow w(y)
\]
for almost every $y\in \D$ and in $L^1(\D)$. Together with the uniform bound
\[
\left| W_{\phi^{k \lceil \frac{n}{k} \rceil}} (x,y) - W_{\phi^n}(x,y) \right| = \left| \sum_{j=n}^{k \lceil \frac{n}{k} \rceil - 1} W_{\phi}(\phi^j(x),\phi^j(y)) \right| \leq k \|W_{\phi}\|_{\infty},
\]
this implies that the sequence of functions
\[
y \mapsto \frac{W_{\phi}^n(x,y)}{n}
\]
converges to $W_{\phi}^{\infty}(x,y):= w(y)/k$ for almost every $y\in \D$ and in $L^1(\D)$.
\end{rem}


\section{Main results}

Let $\phi\in \text{Diff}_c(\D, \omega_0)$ and let $\phi^t$ be a Hamiltonian isotopy such that $\phi^0=\mathrm{id}$ and $\phi^1=\phi$. We fix $x$ in the interior of $\D$, a point $e\in \partial \D$ and we define the surface $S^e_\phi(x)\subset \D\times [0,1]$ as in (\ref{reparS}). Given $y\in \D \setminus \{x\}$ such that the curve $\Gamma_\phi(y) = \{(\phi^t(y),t) \mid t\in [0,1]\}$ meets $S^e_\phi(x)$ transversally, we denote by $I_\phi^e(x,y)$ the intersection number of $\Gamma_\phi(y)$ with $S^e_\phi(x)$, as in Section \ref{wiidd}. As observed in Remark \ref{ae}, $I_\phi^e(x,\cdot)$ is defined on a full measure subset of $\D$.

\begin{prop}
\label{actionvsintersection}

The action of $x$ with respect to a primitive $\lambda$ of $\omega_0$ and the intersection numbers $I^e_\phi(x,\cdot)$ are related by the identity
\begin{equation}\label{actionvsInt}
    a_{\phi,\lambda}(x)=\int_\D I_\phi^e(x,y)\, \omega_0(y)-\int_{[e,x]}\lambda+\int_{[e,\phi(x)]}\lambda,
\end{equation}
where $[e,x]$ and $[e,\phi(x)]$ denote oriented line segments in $\D$.
\end{prop}

\begin{proof}Let $H_t$ be the time-dependent compactly supported Hamiltonian that defines the isotopy $\phi^t$ and denote by $X_t$ its Hamiltonian vector field. We lift the differential forms $\lambda$ and $\omega_0$ and the vector field $X_t$ to corresponding objects on the three-manifold $\D\times [0,1]$:
\begin{align*}
    \Tilde{\lambda}&:=\lambda+H_t dt, \\
    \Tilde{\omega}_0 &:=d\tilde{\lambda}=\omega_0+dH_t\wedge dt, \\
    \tilde{X}&:=X_t+\partial_t.
\end{align*}
Note that $\tilde{X}$ is an autonomous vector field on $\D\times [0,1]$ and
\begin{align}
    \imath_{\tilde{X}}{\Tilde{\omega}_{0}} &= \imath_{(X_t+\partial_t)}[ \omega_0+dH_t \wedge dt)]\nonumber\\
                    &= \imath_{X_t}\omega_0+\imath_{X_t}(dH_t \wedge dt)+\imath_{\partial_t}\omega_0+\imath_{\partial_t}(dH_t \wedge dt)\nonumber\\
                    &=dH_t+0 + 0 - dH_t = 0.\label{E:flow invariant}
\end{align}
In particular, using also the fact that $\Tilde{\omega}_0$ is closed, we have
\[
L_{\tilde{X}} \Tilde{\omega}_0 = d \imath_{\tilde{X}} \Tilde{\omega}_0  + \imath_{\tilde{X}} d \Tilde{\omega}_0  = 0,
\]
and hence $\Tilde{\omega}_0$ is invariant with respect to the local flow of $\tilde{X}$, which we denote by $\tilde{\phi}$. Note that
\[
\tilde{\phi}^t(z,0) = (\phi^t(z),t) \qquad \forall z\in \D, \; t\in [0,1].
\]
{Note that (\ref{E:flow invariant}) together with the fact that $\tilde{\omega}_0$ is closed, also shows that $\tilde{\omega}_0$ is invariant with respect to any reparametrization of this local flow.}
The map 
\[
\rho: S^e_\phi(x) \rightarrow \D, \qquad (z,t) \mapsto (\phi^t)^{-1}(z),
\]
from Section \ref{wiidd} maps each $(z,t)\in S^e_\phi(x)$ into the unique point $y\in \D$ such that $(y,0)$ is on the backward orbit of $(z,t)$ by $\tilde{\phi}$. In other words, this map has the form
\begin{equation*}
    \rho(p)=\pi \circ \tilde{\phi}^{-\tau(p)}(p),
\end{equation*}
where $\pi: \D \times [0,1]\rightarrow \D$ and $\tau: \D \times [0,1] \rightarrow [0,1]$ are the two projections. 
We set $\psi(p) := \tilde{\phi}^{-\tau(p)}(p)$ and we compute its differential
\[
d\psi(p)[u]=d\tilde{\phi}^{-\tau(p)}(p)[u]- d\tau(p)[u]\tilde{X}(\tilde{\phi}^{-\tau(p)}(p)), \qquad \forall u\in T_p S^e_\phi(x).
\]
Notice that $\rho=\pi\circ \psi$ and $\pi^*\omega_0=\tilde{\omega}_0-dH_t\wedge dt$. These two identities let us compute the pullback of the form $\omega_0$ by the map $\rho$ at some point $p\in S^e_{\phi}(x)$: for every $u,v\in T_p S^e_{\phi}(x)$ we have
\begin{align*}
    \rho^*\omega_0(p)[u,v]  &=(\psi^*\pi^*\omega_0)(p)[u,v]\\
                            &=\psi^*(\tilde{\omega}_0-dH_t\wedge dt)(p)[u,v]\\
                            &=\psi^*\tilde{\omega}_0(p)[u,v]-\psi^*(dH_t\wedge dt)(p)[u,v]\\
                            &=\tilde{\omega}_0(\psi(p))\big[d\psi(p)[u],d\psi(p)[v]\big]{-}dH_0(\psi(p)) \wedge dt \big[d\psi(p)[u],d\psi(p)[v]\big].
\end{align*}
The fact that the flow $\tilde{\phi}$ and reparametrizations of it preserve the form $\Tilde{\omega_0}$ implies that $\psi^*\Tilde{\omega}_0 = \Tilde{\omega}_0$, and hence the first term in the last expression equals $\tilde{\omega}_0(p)[u,v]$. On the other hand, the second term in the last expression vanishes, because the vectors $d\psi(p)[u]$ and $d\psi(p)[v]$ belong to the space $T_{\psi(p)}\D \times \{0\}$, on which $dt$ vanishes. We conclude that
\[
\rho^*\omega_0=\tilde{\omega}_{0} |_{S^e_\phi(x)}.  
\]
By Stokes' theorem and (\ref{actionH}) we obtain the identity
\begin{equation}{\label{1}}
\begin{split}
\int_{S^e_\phi(x)}\rho^*\omega_0 &=\int_{S^e_\phi(x)}\tilde{\omega}_0 = \int_{\partial S^e_\phi(x)}\tilde{\lambda}\\ &= \int_{\partial S^e_\phi(x)} (\lambda + H_t \, dt) = a_{\phi,\lambda}(x)+\int_{[e,x]}\lambda-\int_{[e,\phi(x)]}\lambda. \end{split}
\end{equation}
On the other hand, equation (\ref{degreeProperty}) and Lemma \ref{transverseregular} yield
\begin{equation}{\label{2}}
\int_{S^e_\phi(x)}\rho^*\omega_0=\int_\D\deg(\rho,S^e_\phi(x),y)\, \omega_0(y)=\int_\D I_\phi^e(x,y)\, \omega_0(y).
\end{equation}
The desired identity (\ref{actionvsInt}) follows from  (\ref{1}) and (\ref{2}).

\end{proof}

\begin{rem}
In the particular case of a fixed point $x\in \D$ of $\phi$, the formula of Proposition \ref{actionvsintersection} reduces to the identity
\[
a_{\phi,\lambda}(x)=\int_\D I_\phi^e(x,y)\omega_0(y).
\]

\end{rem}

We can finally prove Theorem \ref{main-thm} from the Introduction.

\begin{proof}[Proof of Theorem \ref{main-thm}]
By applying Proposition \ref{actionvsintersection} to the map $\phi^n$ we obtain the identity
\begin{equation*}
    \frac{a_{\phi^n,\lambda}(x)}{n}=\frac{1}{n}\int_\D I_{\phi^n}^{e}(x,y)\, \omega_0(y)-\frac{1}{n}\int_{[e,x]}\lambda+\frac{1}{n}\int_{[e,\phi^n(x)]}\lambda.
\end{equation*}
for every $x$ in $\D$. The last two integrals are uniformly bounded in $x\in \D$ and $n\in \mathbb{N}$. Together with the bound from Proposition \ref{Bound1} applied to $\phi^n$, we deduce the bound
\begin{equation}
\label{est}
    \frac{a_{\phi^n,\lambda}(x)}{n}-\frac{1}{n}\int_\D W_{\phi^n}(x,y)\, \omega_0(y) = O \left( \frac{1}{n} \right)
\end{equation}
uniformly in $x\in \D$. The set $\Omega$ consisting of all points $x\in \D$ for which the sequence $a_{\phi,\lambda}(x)/n$ converges to $a_\phi^{\infty}(x)$ and the integrals of $W_{\phi^n}(x,\cdot) \omega_0/n$ converge to the integral of $W_\phi^{\infty}(x,\cdot)$ has full measure in $\D$. Taking the limit in (\ref{est}) we obtain the identity
\begin{equation}
\label{last}
a_\phi^{\infty}(x) = \int_\D W_\phi^{\infty}(x,y)\, \omega_0(y)
\end{equation}
for every $x\in \Omega$.   The set $\Omega$ contains all periodic points of $\phi$ thanks to the identity (\ref{actionPeriodic}) and Remark \ref{periodic}.
\end{proof}

Fathi's formula stated as Corollary \ref{cor:fathi} in the Introduction now follows by integrating (\ref{last}) over the full measure set $\Omega$ and using the identities (\ref{BirkAction}) and (\ref{intW}).


\end{document}